\documentclass{article}
\usepackage{graphicx}
\usepackage{url}
\usepackage{latexsym,amscd,amssymb, amsmath,graphicx, amsthm,xcolor,mathrsfs,mathtools,amsfonts}   
\usepackage[margin=1in]{geometry}
\usepackage[colorinlistoftodos]{todonotes}

\newtheorem{theorem}{Theorem}[section]
\newtheorem{proposition}[theorem]{Proposition}
\newtheorem{corollary}[theorem]{Corollary}
\newtheorem{lemma}[theorem]{Lemma}

\theoremstyle{definition}
\newtheorem{definition}[theorem]{Definition}
\newtheorem{observation}[theorem]{Observation}

\newtheorem{problem}[theorem]{Problem}
\newtheorem{example}[theorem]{Example}

\newcommand{\fkS}{\mathfrak{S}}
\newcommand{\Pf}{\mathrm{Pf}}
\newcommand{\NSh}{\mathrm{NShPf}}
\newcommand{\OSh}{\mathrm{OShPf}}

\newcommand{\ch}{\mathtt{ch}}
\newcommand{\cl}{\mathscr{C}}

\newcommand{\sort}{\mathtt{sort}}

\newcommand{\Krew}{\mathrm{Krew}}
\newcommand{\OKrew}{\mathrm{OKrew}}
\newcommand{\Sym}{\mathrm{Sym}}
\newcommand{\SymP}{\mathrm{SymP}}
\newcommand{\Gar}{\mathrm{Gar}}
\newcommand{\area}{\mathrm{area}}
\newcommand{\bounce}{\mathrm{bounce}}
\newcommand{\dinv}{\mathrm{dinv}}
\newcommand{\stat}{\mathrm{stat}}
\newcommand{\sh}{\mathtt{sh}}

\newcommand{\CC}{\mathbb{C}}
\newcommand{\NN}{\mathbb{N}}

\newcommand{\C}

\title{Odd Shifted Parking Functions}
\author{Zachary Hamaker, Jesse Kim}
\date{\today}

\begin{document}

\maketitle

\abstract{
Stanley recently introduced the shifted parking function symmetric function $SH_n$, which is the shiftification of Haiman's parking function symmetric function $PF_n$.
The function $SH_n$ lives in the subalgebra of symmetric functions generated by odd power sums.
Stanley showed how to expand $SH_n$ into the  $V$--basis of this algebra, which is indexed by partitions with all parts odd and is analogous to the complete homogeneous (or elementary) basis of symmetric functions.
We introduce \emph{odd shifted parking functions} to give combinatorial and representation-theoretic realizations of the $V$--expansion of $SH_n$, resolving the main open problem in his paper.
Further, we present two representation-theoretic realizations of shiftification allowing us to interpret $SH_n$ as the spin character of a projective representation.
We conclude with further directions, including a relationship between $SH_n$ and Haglund's $(q,t)$--Schr\"oder theorem.}

\section{Introduction}
\label{s:intro}

A \emph{parking function} of size $n$ is a tuple $(p_1,\dots,p_n)$ of positive integers so that the $i$th smallest entry is at most $i$.
The symmetric group $\fkS_n$ acts on parking functions of size $n$ by permuting their entries.
The \emph{parking function symmetric function} $PF_n$ is the Frobenius character of this $\fkS_n$--action on the set $\Pf(n)$ of parking functions of size $n$.
Parking functions have been studied extensively in the combinatorics community (see~\cite{yan2015parking} for a recent survey), in large part due to Haiman's observation that a natural bi-graded version of $PF_n$ is the Frobenius character for the space of diagonal harmonics~\cite{haiman2001hilbert,haiman1994conjectures}.

The \emph{shiftification map} $\sh$, sometimes denoted $\Theta$, is an operator on symmetric functions $\Sym$ defined on power sums by $\sh(p_{2k-1}) = 2p_{2k-1}$ and $\sh(p_{2k}) = 0$.
The image of shiftification is the subalgebra $\SymP$ generated by odd power sums.
The algebra $\SymP$ has several bases, for instance $\{p_\lambda: \lambda\ \mathrm{odd}\}$ where $\lambda$ is \emph{odd} if each $\lambda_i$ is.
The \emph{Schur--$P$ functions} $P_\lambda$, indexed by partitions with all parts distinct, also form a basis analogous to Schur functions.
For $k$ a positive integer, single part Schur--$P$ functions satisfy the relation
\begin{equation}
    \label{eq:R-relation}
    P_{2k} + \sum_{i=1}^{k-1} P_{2i}\cdot P_{2(k-i)} = \sum_{i=1}^k P_{2i-1}\cdot P_{2(k-i)+1}.
\end{equation}
Therefore, with $V_\lambda = \prod_{i} P_{\lambda_i}$ we see $\{V_\lambda: \lambda\ \mathrm{odd}\}$ is a basis of $\SymP$ as well.

Let $\bigwedge(\CC^n)$ be the exterior  algebra and $\fkS_n$ act on it by permuting standard basis elements.
We give an apparently novel representation-theoretic realization of shiftification, closely related to a classical construction appearing in~\cite{stembridge} using the basic spin representation.

\begin{proposition}
    \label{p:shiftification}
    Let $M$ be an $\fkS_n$--module with Frobenius character $f$. 
    Then the Frobenius character of $M \otimes \bigwedge(\CC^n)$ is $\sh(f)$.
    
\end{proposition}

In~\cite{stanley2024shifted}, Stanley applied shiftification to the parking function symmetric function $PF_n$, constructing the \emph{shifted parking function symmetric function} $SH_n$.
Using symmetric function theory methods, Stanley showed how to expand $SH_n$ into the power sum, Schur--$P$ and $V$--bases of $\SymP$.
Additionally, he introduced \emph{naive shifted parking functions}, which are pairs $(p,\sigma)$ with $p \in \Pf(n)$ and $\sigma \in \{-1,1\}^n$.
Let $\NSh(n)$ be the set of naive shifted parking functions.
By identifying $\sigma$ with a basis element of the exterior algebra $\bigwedge^n(\CC)$, Proposition~\ref{p:shiftification} shows the $\fkS_n$--action on $\CC[\NSh(n)] \cong \CC[\Pf(n)] \otimes \bigwedge^n(\CC)$ has Frobenius character $SH_n$ (see Corollary~\ref{c:naive-frob}).
This gives a combinatorial/representation-theoretic explanation for Stanley's expansion~\cite[Thm. 3.4]{stanley2024shifted} of $SH_n$ into the spanning set $\{V_\lambda:\lambda \vdash n\}$.

Next, we introduce \emph{odd shifted parking functions}, which are triples $(p,\sigma,\tau)$.
Here $p$ is a parking function whose shape is odd, $\sigma$ is a sign function satisfying certain parity constraints and $\tau$ is a non-crossing partial matching of the integers occurring in $p$.
Then $\fkS_n$ acts on the set of odd shifted parking functions $\OSh(n)$ by permuting $p$ and $\sigma$.
A \emph{sorted odd shifted parking function} is a triple $(p,\overline{\sigma},\tau)$ where $p$ is a weakly increasing parking function whose shape is odd, $\tau$ is a non-crossing matching as above and $\overline{\sigma}$ assigns a sign to each integer occurring in $p$ satisfying certain constraints determined by $\tau$.
These objects are defined in detail in Section~\ref{s:main}.
Our main result is:
 
\begin{theorem}
    \label{t:odd}
    The Frobenius character of the $\fkS_n$--action on $\OSh(n)$ is $SH_n$.
    Moreover,
    \begin{equation}
        \label{eq:t-odd}
    SH_n = \sum_{(p,\overline{\sigma},\tau)\ \mathrm{sorted}} V_{\mathrm{shape}(p)}.
    \end{equation}
\end{theorem}

Theorem~\ref{t:odd} gives a combinatorial description of the $V$--basis expansion of $SH_n$, resolving Stanley's main open problem~\cite{stanley2024shifted}.
Stanley also asks for a projective representation on shifted parking functions whose spin character is $SH_n$, concluding `probably this is too much to hope for.'
Adapting work of Stembridge~\cite[\S9]{stembridge}, a construction similar to Proposition~\ref{p:shiftification} using the Clifford algebra in place of the exterior algebra allows us to extend our $\fkS_n$ action on $\OSh(n)$ to a projective representation (see Corollary~\ref{c:odd-projective}).
We show in~\eqref{eq:t-sh} and Proposition~\ref{p:t-odd-sh} that the naive $V$--expansion and Theorem~\ref{t:odd} both extend to a $t$--graded version.

Let $R_{2k}$ equal~\eqref{eq:R-relation}, $R_{2k-1} = P_{2k-1}$ and $R_{\lambda} = \prod_{i} R_{\lambda_i}$. 
Our proof of Theorem~\ref{t:odd} proceeds by grouping terms from the naive shifted parking function description of $SH_n$ using the left-hand-side of~\eqref{eq:R-relation} so that even parts now contribute to $R_{2k}$ terms.
This gives a combinatorial expansion of $SH_n$ into the spanning set $\{R_\lambda:\lambda \vdash n\}$.
We then show our sorted odd shifted parking functions can be obtained by splitting $R_{2k}$'s using the right-hand-side of~\eqref{eq:R-relation} to obtain~\eqref{eq:t-odd}, and in turn the Frobenius character result.

\subsection*{Organization} The remainder of the paper is organized as follows.
In Section~\ref{s:background}, we introduce relevant background.
This includes properties of parking functions plus their naive shifted counterparts, symmetric functions with a focus on $SH_n$ and a brief introduction to projective representations including our representation theoretic explanations of shiftification using exterior algebras and Clifford algebras.
Our main results, combinatorial and representation-theoretic descriptions of the $V_\lambda$ expansion of $SH_n$, are presented in Section~\ref{s:main}.
We conclude with some smaller results and several directions for further exploration in Section~\ref{s:final}, both enumerative and algebraic.

\subsection*{Acknowledgements}
We thank Leigh Foster for suggesting the term \emph{garage}, Sean Griffin for helping us understand the history of diagonal harmonics, Brendon Rhoades for sharing the $\fkS_n$--action on $\NSh(n)$, Richard Stanley for encouragement and John Stembridge for helpful remarks on projective representations.

\section{Background}
\label{s:background}
Let $[n] = \{1,2,\dots,n\}$ and $\fkS_n$ be the permutations of $[n]$.
Note that $\fkS_n$ acts on tuples $(a_1,\dots,a_n)$ by $w \cdot (a_1,\dots,a_n) = (a_{w(1)},\dots,a_{w(n)})$.
Recall an integer partition $\lambda = (\lambda_1,\dots,\lambda_k)$ is a weakly descending tuple of non-negative integers with \emph{size} $\sum \lambda_i$ and length $\ell(\lambda) = \max \{i:\lambda_i > 0\}$.
Let $m_i(\lambda) = |\{j:\lambda_j = i\}$.
We say $\lambda$ is \emph{odd} if $\lambda_i$ is odd when positive and \emph{strict} if $\lambda_1 > \dots > \lambda_{\ell(\lambda)}$.

\subsection{Parking functions}

For $a = (a_1,\dots,a_n) \in \mathbb{R}^n$, let $\sort(a) = \pi \cdot a$ where $\pi \in \fkS_n$ satisfies $a_{\pi(1)} \leq \dots \leq a_{\pi(n)}$.
As stated in the introduction, a \emph{parking function} of \emph{size} $n$ is a tuple $p$ of $n$ positive integers so that for $i \in [n]$ the $i$th entry in $\sort(p)$ is at most $i$.
We say $p$ is \emph{sorted} if $p = \sort(p)$.
Let $\Pf(n)$ be the set of all parking functions of size $n$ and $\overline{\Pf}(n)$ be the set of sorted parking functions of size $n$.
It is a classical result that $|\Pf(n)| = (n{+}1)^{n-1}$ and $|\overline{\Pf}(n)|$ is the $n$th Catalan number $c_n = \frac{1}{n+1}\binom{2n}{n}$.

The \emph{content} of $p \in \Pf(n)$ is $\alpha(p) = (\alpha_1(p),\dots,\alpha_n(p))$ where 
\[
\alpha_k(p) = |\{j \in [n]:p_j = k\}|,
\]
and the \emph{shape} of $p$ is the partition $\lambda(p)$ obtained by reversing $\sort(\alpha)$.
For example, the parking function $p = (4,1,1,3,6,3)$ has $\sort(p) = (1,1,3,3,4,6)$, $\alpha(p) = (2,0,2,1,0,1)$ and $\lambda(p) = (2,2,1,1)$.

Introduced in~\cite[\S 4]{stanley2024shifted}, a \emph{naive shifted parking function} of \emph{size} $n$ is a pair of $n$--tuples $(p,\sigma)$ where $p$ is a parking function and $\sigma \in \{-1,1\}^n$.
Define $\sort(p,\sigma) = (\sort(p),\overline{\sigma})$ where 
\[
\overline{\sigma}_k = \begin{cases}
    \prod_{p_i = k} \sigma_i & \alpha_k(p) > 0\\
    0 & \alpha_k(p) = 0
\end{cases}
\]
is the parity of $\sigma_i$'s where $p_i = k$ when this quantity is defined.
A \emph{sorted naive shifted parking function} of size $n$ is a pair $(p,\overline{\sigma})$ where $p$ is a sorted parking function and $\overline{\sigma} \in \{-1,0,1\}^n$ where $\overline{\sigma}_k = 0$ if and only if $\alpha_k(p) = 0$.
The naive shifted parking function $(p,\sigma)$ inherits its content and shape from $p$.
Let $\NSh(n)$ be the set of naive shifted parking functions of size $n$ and $\overline{\NSh}(n)$ be the set of sorted naive shifted parking functions of size $n$.
It is easy to see $|\NSh(n)| = 2^n (n{+}1)^{n-1}$, while Stanley observes $|\overline{\NSh}(n)|$ is the $n$th large Schr\"oder number~\cite[A006318]{oeis}.
We present a bijective proof of this fact in Section~\ref{ss:final-enumerative}.

\subsection{Symmetric Functions}

Let $\Sym$ denote the algebra of symmetric functions.
We refer the reader to~\cite{macdonald1998symmetric} or~\cite{sagan2013symmetric} for descriptions of the power sum, elementary, homogeneous, Schur symmetric functions and the Hall inner product.

Our results concern $\SymP = \langle p_{2k+1}:k \in \NN\rangle$, the subalgebra of \emph{odd symmetric functions} in $\Sym$.
The \emph{shiftification operator} $\sh: \Sym \to \SymP$ is defined by $h_k \mapsto 2P_k$ where
\begin{equation}
    \label{eq:P-def}
P_k = \frac{1}{2}\sum_{i=0}^k h_ie_{k-i} = \sum_{\substack{i=0\\i \textrm{ odd}} }^k h_ie_{k-i} = \sum_{\substack{i=0\\i \textrm{ even}}}^k h_ie_{k-i},
\end{equation}
or equivalently where $p_{2k-1} \mapsto 2p_{2k+1}$ and $p_{2k} \mapsto 0$ for $k$ a positive integer.
Note $\sh(e_k) = (-1)^k \cdot 2P_k$.
We observe that shiftification is self-adjoint with respect to the Hall inner product.
\begin{proposition}
    \label{p:self-adjoint}
    For $f,g \in \Sym$, $\langle \sh f,g\rangle = \langle f, \sh g\rangle$.
\end{proposition}

\begin{proof}
    Since the power sum basis is orthogonal with respect to $\langle \cdot,\cdot\rangle$ and $\sh$ is diagonal with respect to the power sum basis the result follows.
\end{proof}

Define the generating functions
\[
A = P_1 t+P_3t^3+P_5t^5 + \dots \quad \mbox{and} \quad B = P_2 t^2+P_4 t^4 + P_6 t^6 + \dots,
\]
which by~\cite[Lem~3.1]{stanley2024shifted} satisfy the relation
\[
B = \frac{-1 + \sqrt{1 + 4A^2}}{2},
\]
or equivalently $A^2 = B^2 + B$.
Taking the coefficient of $t^{2k}$ gives
\begin{equation}
    \label{eq:P-relation}
    \sum_{i=1}^k P_{2i} \cdot P_{2k-2i} = \sum_{i=1}^k P_{2i-1} \cdot P_{2k-2i+1},
\end{equation}
which is~\ref{eq:R-relation}.
For $\lambda = (\lambda_1,\dots,\lambda_k)$ a partition, define $V_\lambda = \prod_{i=1}^k P_{\lambda_i}$.
From~\eqref{eq:P-relation} we see the $V_\lambda$'s are not linearly independent.
However, $\{V_\lambda:\lambda\ \vdash n\ \text{odd}\}$ is a basis for $\SymP$.
There are other bases for $\SymP$, most notably the Schur-$P$ and Schur-$Q$ functions indexed by strict partitions, which figure less prominently in our work.

The \emph{parking function symmetric function} is
\begin{equation}
    \label{eq:parking-symm}
    PF_n = \sum_{p \in \overline{\Pf}(n)} h_{\lambda(p)} = \sum_{\lambda \vdash n} \frac{1}{(n-\ell(\lambda)+1)}\binom{n}{m_1(\lambda),\dots,m_n(\lambda),n-\ell(\lambda)} h_\lambda .
\end{equation}
The coefficient of $h_\lambda$ in the second sum is the \emph{Kreweras number} $\Krew(\lambda)$.
We are interested in understanding the \emph{shifted parking function symmetric function} $SH_n = \sh(PF_n)$.
By applying $\sh$ to the right-hand-side of~\eqref{eq:parking-symm}, we have
\begin{equation}
    \label{eq:naive-V}
    SH_n = \sum_{p \in \overline{\Pf}(n)} 2^{\ell(\lambda(p))} V_{\lambda(p)}  = \sum_{\lambda \vdash n} \frac{2^{\ell(\lambda)}}{(n-\ell(\lambda)+1)}\binom{n}{m_1(\lambda),\dots,m_n(\lambda),n-\ell(\lambda)} V_\lambda.
\end{equation}
Since there are $2^{\ell(\lambda(p))}$ choices for $\overline{\sigma}$ associated to each sorted $p$, we also have
\[
SH_n = \sum_{(p,\overline{\sigma}) \in \overline{\NSh}(n)} V_{\lambda(p)}.
\]

Since we are summing over $V_\lambda$ of every shape, the expansion in~\eqref{eq:naive-V} is not unique.
In~\cite{stanley2024shifted}, Stanley used generating function arguments to give several expansions of $SH_n$ into bases of $\SymP$.
Most notably for our purposes, he showed:
\begin{theorem}[{\cite[Thm. 3.4]{stanley2024shifted}}]
    \label{t:odd-v}
    For every $n$, $SH_n = \sum_{\lambda \vdash n \ \text{odd}} \OKrew(\lambda) V_\lambda$ with
    \[
    \OKrew(\lambda) = \frac{2^\ell(\lambda)}{\ell(\lambda)!}\binom{\ell(\lambda)}{m_1(\lambda),m_3(\lambda),\dots} (n+\ell(\lambda)-1)(n+\ell(\lambda)-3)...(n-\ell(\lambda)+3) V_\lambda.
    \]
\end{theorem}

For $\lambda$ odd, we call $\OKrew(\lambda)$ the \emph{odd Kreweras number} of $\lambda$.

\subsection{Frobenius Characters and Spin Characters}

Let $Z_n$ denote the algebra of class functions of $\fkS_n$. The {\em Frobenius characteristic map} is the linear isomorphism $\ch:Z_n \to \Sym$ defined by
\begin{equation}
    \label{eq:Frob}
    \ch(\chi) = \sum_{\lambda \vdash n} z_\lambda^{-1} \chi(\pi_\lambda)p_\lambda.
\end{equation}
Here $\pi_\lambda$ is a permutation with cycle type $\lambda$ and $z_\lambda$ is the size of the stabilizer of $\pi_\lambda$ under conjugation.
The \emph{Frobenius character} of a $\mathbb{C}[\fkS_n]$-module $V$ is the image of the character of $V$ under $\ch$.

Under the Frobenius characteristic map, addition in $\Sym$ corresponds to direct sum of $\CC[\fkS_n]$-modules and multiplication in $\Sym$ corresponds to the {\em induction product} of $\CC[\fkS_n]$-modules. Given a composition $\alpha = (\alpha_1, \dots, \alpha_k)$ of $n$, the Young subgroup corresponding to $\alpha$ is
\begin{equation}
    \label{eq:youngsubgroup}
    S_\alpha = S_{\alpha_1}\times \cdots \times S_{\alpha_k} ,
\end{equation}
viewed as the subgroup of $\fkS_n$ in which $S_{\alpha_1}$ permutes $\{1,\dots, \alpha_1\}$, $S_{\alpha_2}$ permutes $\{\alpha_1+1, \dots,\alpha_1+\alpha_2\}$, etc. If $W_i$ is a $\fkS_{\alpha_i}$-representation, then
\begin{equation}
    \label{eq:induct-prod}
    \ch(\mathbb{C}[\fkS_n] \otimes_{\mathbb{C}[\fkS_\alpha]}W_{\alpha_1} \otimes \cdots \otimes W_{\alpha_k}) = \ch(W_{\alpha_1}) \cdot \ldots \cdot \ch(W_{\alpha_k})
\end{equation}

Pointwise multiplication in $Z_n$ induces a second product on $\Sym$, called the \emph{Kronecker product} or \emph{inner product}.
By analogy, the \emph{Kronecker product} of two symmetric functions $f= \ch(\chi_V)$ and $g = \ch(\chi_W)$ is 
\[
f *g = \ch(\chi_V\chi_W)
\]
Equivalently, if $V$ and $W$ are $\mathbb{C}[\fkS_n]$-modules with Frobenius characters $f$ and $g$, respectively, then $f*g$ is the Frobenius character of $V\otimes W$.
This product extends linearly to all of $\Sym$.

Much as $\Sym$ encodes the characters of $\fkS_n$ representations, the subalgebra $\SymP$ encodes the character theory of spin representations of the double cover of $\fkS_n$.
We refer readers to $\cite{stembridge}$ for details beyond what we include here (see also~\cite{MR1007885}). 
The \emph{negative-type double cover} of $\fkS_n$ is
\[
\fkS_n^- = \langle s_1,\dots,s_{n-1},z: z^2 =1,\ s_i^2 = z,\ s_is_j = z s_j s_i \textrm{ for } |i-j|\geq 2,\ s_is_{i+1}s_i = zs_{i+1}s_is_{i+1}\rangle.
\]
For each odd partition of $n$ or distinct partition $\lambda \vdash n$ so that $n-\ell(\lambda)$ is odd, there are two conjugacy classes in $\fkS^-_n$, one of which is distinguished as \emph{positive}.
There is one conjugacy class in $\fkS^-_n$ for other partitions of $n$.
See~\cite[\S2]{stembridge} for details.

A \emph{spin representation} is a representation of $\fkS_n^-$ for which $z$ acts by $-1$.
Stembridge defines the {\em spin characteristic map} $\ch'$ from the algebra of spin class functions of $\fkS_n^-$, denoted $Z_n'$, to $\SymP$ by 
\[
\ch'(\psi) = \sum_{\substack{\lambda\vdash n\\\lambda \textrm{ odd}}} z_\lambda^{-1} 2^{\ell(\lambda)/2}\psi(\pi_\lambda)p_\lambda
\]
where $\pi_\lambda \in \fkS_n^-$ is in the positive conjugacy class associated to $\lambda$.

Since every irreducible representation of $\fkS_n^-$ is either a $\fkS_n$--representation or a spin representation, the algebra of $\fkS_n^-$ class functions is  $Z_n \oplus Z_{n}'$. Multiplication on $Z_n \oplus Z_n'$ thus induces a multiplication on $\Sym \oplus \SymP$ via the map $\ch \oplus \ch'$, which is also called the Kronecker product. 
Using Stembridge's notation, for $f \in \SymP$ let $f$ denote $(f,0) \in \Sym \oplus \SymP$ and $f'$ denote $(0,f) \in \Sym \oplus \SymP$.

\begin{proposition}[{\cite[Proof of Prop. 9.1]{stembridge}}]
\label{p:shift-is-kronecker}
    Let $f \in \Sym$ be homogeneous of degree $n$. Then
    \[f*2P_n = \sh f \quad  \text{and} \quad f* 2P_n' = \sh f'.
    \]
\end{proposition}

Proposition~\ref{p:shift-is-kronecker} gives two alternate definitions of shiftification.
We realize these definitions via general constructions for both $\fkS_n$--modules and $\fkS^-_n$--modules.
The former construction is described in Proposition~\ref{p:shiftification}, while the latter is implicit in~\cite{stembridge} and related work.
As a consequence of these constructions, we can modify $\CC[\Pf_n]$ to build modules whose characters' images under $\ch$ and $\ch'$, respectively, are $SH_n$.

We begin with an observation that an exterior algebra has character $2P_n$.
\begin{observation}
    \label{o:exterior}
    Let $\fkS_n$ act on $\mathbb{C}^n$ by permuting basis elements $\{\beta_1, \dots, \beta_n\}$. Then the induced action on $\bigwedge (\mathbb{C}^n)$ has Frobenius character $2P_n$. Furthermore, the even subalgebra $\bigwedge^{\textrm{even}}(\CC^n) := \bigoplus_k \bigwedge^{2k}(\mathbb{C}^n)$ and odd part $\bigwedge^{\textrm{odd}}(\CC^n) := \bigoplus_k \bigwedge^{2k+1}(\mathbb{C}^n)$ are $\fkS_n$-submodules with character $P_n$.
\end{observation}

\begin{proof}
    The degree $k$ component is
    \[
    \bigwedge\nolimits^k (\mathbb{C}^n) \cong_{\mathbb{C}[\fkS_n]} \CC[\fkS_n] \otimes_{\CC[\fkS_k \times \fkS_{n-k}]} \mathbb{C}[\beta_1 \wedge \cdots \wedge \beta_k]
    \]
    and thus has character $e_kh_{n-k}$ so the result follows by~\eqref{eq:P-def}.\end{proof}

We are now prepared to prove Proposition~\ref{p:shiftification}.
\begin{proof}[Proof of Proposition~\ref{p:shiftification}]
    Apply Observation~\ref{o:exterior} and the first equality in Proposition~\ref{p:shift-is-kronecker}.
\end{proof}

Identify $\mathbb{C}[\NSh]$ with the $\fkS_n$-module $\mathbb{C}[\Pf] \otimes \bigwedge (\mathbb{C}^n)$ via
    \[
    (p, \sigma) \mapsto p \otimes (v_{i_1} \wedge \cdots \wedge v_{i_k})
    \]
    where $\{i_1 \leq \cdots \leq i_k\}$ is the set of indices $i$ for which $\sigma_i = -1$.
Applying Proposition~\ref{p:shiftification} to $\CC[\NSh_n]$, we have the following corollary.

\begin{corollary}
\label{c:naive-frob}
     The Frobenius character of $\mathbb{C}[\NSh] \cong \mathbb{C}[\Pf] \otimes \bigwedge (\mathbb{C}^n)$ is $SH_n$.
\end{corollary}

The resulting $\fkS_n$--action on $\NSh$ preserves the sort of a naive shifted parking function and thus $\CC[\NSh]$ decomposes into submodules.

\begin{lemma}
\label{l:naive-decomp}
    Let $(p, \overline{\sigma}) \in \overline{\NSh}$ and $N_{(p, \overline{\sigma})}$ be the set of naive shifted parking functions whose sort is $(p, \overline{\sigma})$.
    Then $\CC[N_{(p,\overline{\sigma})}]$  is an $\fkS_n$--submodule of $\CC[\NSh]$ with character $V_{\lambda(p)}$.
\end{lemma}

\begin{proof}
Let $\alpha$ be the content of $p$, and let $N'_{(p, \overline{\sigma})} = \{(p', \sigma') \in N_{(p, \overline{\sigma})} \mid p' = p \}$.
Then $\CC[N'_{(p, \overline{\sigma})}]$ is closed under the action of $\fkS_\alpha$, and the $\fkS_n$--closure of $\CC[N'_{(p, \overline{\sigma})}]$ is $\CC[N_{(p, \overline{\sigma})}]$.
As an $\fkS_\alpha$-module, 
    \[
    \CC[N'_{(p, \overline{\sigma})}] \cong \bigwedge\nolimits^{\overline{\sigma}_1}(\CC^{\alpha_1}) \otimes \cdots\otimes \bigwedge\nolimits^{\overline{\sigma}_n} (\CC^{\alpha_n})
    \]
where $\bigwedge^{\overline{\sigma}_i}(\mathbb{C}^{\alpha_i})$ is intended to denote $\bigwedge^{\textrm{even}} (\mathbb{C}^{\alpha_i})$ if $\overline{\sigma}_i = 1$ and $\bigwedge^{\textrm{odd}} (\mathbb{C}^{\alpha_i})$ if $\overline{\sigma}_i = -1$.
Then by~\eqref{eq:induct-prod} and Observation~\ref{o:exterior}, we see
    \[
    \ch(\CC[N_{(p, \overline{\sigma})}]) = \ch( \CC[\fkS_n] \otimes_{\CC[\fkS_\alpha]} \CC[N'_{(p, \overline{\sigma})}]) = \ch\left( \bigwedge\nolimits^{\overline{\sigma}_1}(\CC^{\alpha_1})\right) \cdots \ch\left(\bigwedge\nolimits^{\overline{\sigma}_n}(\CC^{\alpha_n}) \right)  = P_{\alpha_1}\cdots P_{\alpha_n} = V_\lambda.
    \]
\end{proof}

To build an $\fkS_n^-$--module whose spin character is a constant multiple of $SH_n$, we use Clifford algebras in place of exterior algebras.
We follow the convention that the \emph{Clifford algebra} $\cl_n$ is the algebra over $\CC$ with $n$ generators $\xi_1, \dots, \xi_n$ satisfying the relations $\xi_i^2 =1$ and $\xi_i\xi_j = -\xi_j\xi_i$ for $i\neq j$.
The double cover $\fkS_n^-$ embeds into the multiplicative group of $\cl_n$ with 
\[
s_i \mapsto \frac{i}{\sqrt{2}} (\xi_i-\xi_{i+1}).
\]
Since $z = s_i^2$, the reader can check $z \mapsto -1$.
The Clifford algebra acts on itself by left multiplication, so this embedding turns $\cl_n$ into an $\fkS_n^-$--module.
The following is implicit in~\cite[\S4]{stembridge}:
\begin{proposition}
\label{p:char-of-clifford}
    The spin character of $\cl_n$ is $2^{\frac{n}{2}+1}P_n$. 
\end{proposition}

\begin{proof}

For $ I = \{i_1 \leq i_2 \leq \cdots \leq i_m\} \subseteq [n]$, define $\xi_I =  \xi_{i_1}\cdots \xi_{i_m}$.
Then $\{\xi_I \mid I \subseteq [n]\}$ is a basis of $\cl_n$.
For $J \subseteq [n]$, multiplication by $\xi_J$ permutes this basis, with no fixed points unless $J = \varnothing$.
Therefore, the trace of $\xi_J$'s action is 0 unless $J = \varnothing$, in which case $\xi_J = 1$ acts as the identity with trace $2^n$.

    Let $\lambda = (\lambda_1, \dots, \lambda_k)$ be an odd partition, and let $r_j = \sum_{i=0}^{j-1} \lambda_i$. A representative for the positive conjugacy class for $\lambda$ is given by $\pi_\lambda = \pi_1 \cdots \pi_k$ where $\pi_j = s_{r_j}s_{r_j+1}\cdots s_{r_j+\lambda_j -1}$. The image of $\pi_j$ in $\cl_n$ is thus
    \[
    \left(\frac{i}{\sqrt{2}}\right)^{\lambda_j-1} (\xi_{r_j}-\xi_{r_{j+1}})(\xi_{r_j+1}- \xi_{r_j+2})\cdots (\xi_{r_j+\lambda_j-1} - \xi_{r_j + \lambda_j})
    \]
    By the preceding paragraph, the trace of $\pi_j$ is $2^n$ times the coefficient of the constant term, $2^{-\frac{\lambda_j-1}{2}}$. Thus, the trace of $\pi_\lambda$ is $2^{\frac{n + \ell(\lambda)}{2}}$, so
    \[
    \ch'(\cl_n) = \sum_{\substack{\lambda\vdash n\\\lambda \textrm{ odd}}} z_\lambda^{-1} 2^{\ell(\lambda)/2}2^{\frac{n + \ell(\lambda)}{2}}p_\lambda = 2^{\frac{n}{2}+1}P_n
    \]
    where the last equality follows by applying $\sh$ to the identity $h_n = \sum_{\lambda} p_\lambda/z_\lambda$.
\end{proof}
This construction is closely related to one using the spin subalgebra of $\cl_n$ that appears in~\cite{stembridge} that does not include any powers of 2.

If we instead identify $\mathbb{C}[\NSh]$ with the projective $\fkS_n$-module $\mathbb{C}[\Pf] \otimes \cl_n$ via
    \[
    (p, \sigma) \mapsto p \otimes (\zeta_{i_1}  \cdots \zeta_{i_k})
    \]
    where $\{i_1 \leq \cdots \leq i_k\}$ is the set of indices $i$ for which $\sigma_i = -1$, then Proposition~\ref{p:char-of-clifford} and Proposition~\ref{p:shift-is-kronecker} combine to give:
\begin{corollary}
     The spin character of $\mathbb{C}[\NSh] \cong \CC[\Pf] \otimes \cl_n$  is $2^{n/2}SH_n$.
\end{corollary}

\section{Main results}
\label{s:main}

In this section, we show how to modify the combinatorial realization of $SH_n$ via naive shifted parking functions from Corollary~\ref{c:naive-frob} so that it is instead expressed in terms of $V_\lambda$'s with each $\lambda$ odd.
To do so, we introduce a lattice path associated to each naive shifted parking function.
Naive shifted parking functions with the same lattice path are grouped using a novel structure we call a garage.
We then introduce odd shifted parking functions and show they also can be grouped using garages, giving the expansion of $SH_n$ into the $V_\lambda$ basis with $\lambda$'s odd from Theorem~\ref{t:odd-v}.

\subsection{Garages and naive shifted parking functions}
\label{ss:garage}

The key to our argument is to relate naive and (forthcoming) odd shifted parking functions by grouping them via garages, which we now introduce.
Recall for $p$ a parking function that $\alpha_i(p)$ is the number of times $i$ occurs in $p$.

\begin{definition}
\label{d:garage}
For $(p, \overline{\sigma})$ a sorted naive shifted parking function, let $\upsilon \in \{0,1,2\}^n$ with
\begin{equation}
\label{eq:upsilon}
\upsilon_i = \begin{cases}
    0 & \alpha_i(p) = 0;\\
    1 & \alpha_i(p) \ \text{is odd};\\
    2 & \alpha_i(p) > 0\ \text{is even.}
\end{cases}
\end{equation}
Then $(p,\overline{\sigma})$ is a \emph{garage} if for $i<j$ with $\upsilon_i = 2$ and $\upsilon_j = 2$ so that the word $
    \upsilon_{i+1}\upsilon_{i+2} \cdots \upsilon_{j-1} $ has as many $2$'s as $0$'s in any prefix, we then have $\overline{\sigma}_i = \overline{\sigma}_j$.
Let $\Gar(n)$ denote the set of garages of size $n$. 
\end{definition}

The lattice words used to define garages can be interpreted as a lattice path.
More generally, every sorted naive shifted parking function has an associated lattice path.

\begin{definition}
    \label{d:lattice-path}
Let $(p,\overline{\sigma})$ be a sorted naive shifted parking function.
Again define $\upsilon$ as in~\eqref{eq:upsilon}.
Define the \emph{matching path} $L$ associated to $(p,\overline{\sigma})$ by, for each $i$, adding a step according to the following rules:
\begin{itemize}
    \item up-step: $L_i = (0,1)$ if $\upsilon_i = 2$ and $\overline{\sigma}_i = +$ or $\upsilon_i = 0$ and the path is currently below the diagonal;
    \item right-step: $L_i = (1,0)$ if $\upsilon_i = 2$ and $\overline{\sigma}_i = -$ or $\upsilon_i = 0$ and the path is currently above the diagonal;
    \item positive diagonal-step $L_i = (1,1)+$ if $\upsilon_i = 1$;
    \item negative diagonal step: $L_i = (1,1)-$ if $\upsilon_i = 0$ and the path is currently on the diagonal.
\end{itemize}
Let $\mathcal{L}(n)$ be the set of lattice paths associated to naive shifted parking functions of size $n$.
See Figure~\ref{fig:lattice-path} and Example~\ref{ex:lattice-path} for some examples of this construction.
\end{definition}

For garages, the lattice word parity constraint can be summarized as:

\begin{lemma}
    \label{l:garage-path}
A sorted naive shifted parking function $(p,\overline{\sigma})$ is a garage if and only if when the $i$th step in its matching path is towards the diagonal, $\upsilon_i = 0$.
\end{lemma}

To a lattice path $L$ we associate the matching $\tau(L)$ where $i$ and $k$ are matched, i.e., $(i,k) \in \tau(L)$, if $i<k$, $L_i$ and $L_k$ intersect the same diagonal $x -y = d + 1/2$ and no intermediate step $L_j$ does.
Restricted to Dyck paths, this is the standard map to non-crossing matchings.
Note for $(i,k) \in \tau(L)$ that $L_i$ must be moving away from the diagonal, so $\upsilon_i = 2$.
Also, note diagonal steps can never be matched.
Therefore, Lemma~\ref{l:garage-path} says $(p,\overline{\sigma})$ with matching path $L$ is a garage if and only if for each $(i,k) \in \tau(L)$ we see $\upsilon_k = 0$.

The following technical lemma will be useful for verifying certain tuples are parking functions.
\begin{lemma}
    \label{l:path-pigeonhole}
    Let $(p,\overline{\sigma})$ be a sorted naive shifted parking function with matching path $L$ and suppose $(i,k) \in \tau(L)$. Then $\sum_{j=i}^{k-1} \upsilon_{j} \geq k-i$. 
\end{lemma}

\begin{proof}
    If $i$ and $k$ are matched, then there are an equal number of up steps and right steps occurring between them. Furthermore, the path cannot touch the main diagonal between them, so every occurrence of $\upsilon = 0$ moves towards the diagonal. Thus, there are at least as many 2's as 0's in $\{\upsilon_i, \dots, \upsilon_{k-1}\}$ and the result follows.
\end{proof}

Next, we show how garages can be modified to construct other sorted naive shifted parking functions.
\begin{lemma}
    \label{l:modify-naive-matching}
Let $(p,\overline{\sigma})$ be a sorted naive shifted parking function with matching path $L$ so that $(i,k) \in \tau(L)$ with the $2\ell$ total occurrences of $i$ and $k$ in $p$.
For $j \in [\ell]$, construct $(p',\overline{\sigma}')$ from $(p,\overline{\sigma})$ by replacing all occurrences of $i$ and $k$ in $p$ with $2j$ occurrences of $i$ and $2\ell - 2j$ occurrences of $k$, with $\overline{\sigma}'_k = -\overline{\sigma}_i$ if $j \neq \ell$ and $0$ if $j = \ell$.
Then $(p',\overline{\sigma})$ is a naive shifted parking function with matching path $L$.
\end{lemma}

\begin{proof}
As previously observed, we must have $\upsilon_i = 2$ for $i$ and $k$ to be matched.
Therefore, by construction we see the matching path of $(p',\overline{\sigma}')$ is $L$.
All that remains is to show $(p',\overline{\sigma}')$ is also a sorted naive shifted parking function, which follows by Lemma~\ref{l:path-pigeonhole}.
\end{proof}

We say two sorted naive shifted parking functions $(p_1, \overline{\sigma_1})$ and $(p_2, \overline{\sigma_2})$ are {\em garage equivalent} if they have the same matching path $L$ and for each $(i,k) \in \tau(L)$, the total number of occurrences of $i$ and $k$ in $(p_1, \overline{\sigma_1})$ is the same as the total number of occurrences of $i$ and $k$ in $(p_2, \overline{\sigma_2})$.
As previously claimed, garage equivalence partitions the set of sorted naive shifted parking functions.

\begin{lemma}
    \label{l:one-garage-equiv}
    Every sorted naive shifted parking function is garage equivalent to exactly one garage.
\end{lemma}

\begin{proof}
    This follows immediately from Lemmas~\ref{l:garage-path} and \ref{l:modify-naive-matching}.
\end{proof}

Define $\varphi:\NSh(n) \to \Gar(n)$ by mapping each naive shifted parking function to the garage that its sort is garage equivalent to.
Lemma~\ref{l:one-garage-equiv} shows $\varphi$ is well-defined for all $n$. 

From the introduction, recall for $\lambda = (\lambda_1,\dots,\lambda_k)$ that $R_\lambda = R_{\lambda_1} \cdot \ldots \cdot R_{\lambda_k} $ where 
\begin{equation}
    R_m = \begin{cases}
        P_m & m\ \text{odd},\\
        \sum_{i=1}^{m/2} P_{2i} P_{m - 2i} & m\ \text{even}.
    \end{cases}
\end{equation}
\begin{proposition}
\label{p:garage-R}
    For $(p, \overline{\sigma}) \in \Gar(n)$ with $\lambda(p) = \lambda$, $\mathbb{C}[\varphi^{-1}(p, \overline{\sigma})]$ is an $\fkS_n$-submodule of $\mathbb{C}[\NSh(n)]$ whose Frobenius character is
    \[
    \ch(\mathbb{C}[\varphi^{-1}(p, \overline{\sigma})]) = R_\lambda
    \]
\end{proposition}

\begin{proof}
    Note that $\varphi$ factors through the sort map, so we can define $\varphi = \overline{\varphi} \circ \sort$.
    By Lemma~\ref{l:naive-decomp} it suffices to show that $\overline{\varphi}^{-1}((p, \overline{\sigma}))$ is in bijection with terms in the expansion of $R_\lambda$ when $R_{2j}$ is replaced by $\sum_{i=1}^{j} P_{2i} P_{2j - 2i}$.
    Let $L$ be the lattice path of $(p,\overline{\sigma})$ and let $(i,k) \in \tau(L)$ so $\alpha_i(p) + \alpha_k(p) = 2j$.
    By Lemma~\ref{l:modify-naive-matching}, we can modify $(p,\overline{\sigma})$ to produce a sorted naive shifted parking function $(p',\overline{\sigma}')$ with $\alpha_i(p') = 2\ell$ and $\alpha_k(p') = 2j-2\ell$ for any $\ell \in [j]$ and $\overline{\sigma}' = \overline{\sigma}$ unless $\ell = j$, in which case $\overline{\sigma}'_k = 0$. Each choice of $\ell$ corresponds to one term in $R_{2j} = \sum_{i=1}^{j} P_{2i} P_{2j - 2i}$.
    This can be done concurrently for each matched pair, so the result follows. 
\end{proof}

\begin{example}
\label{ex:lattice-path}
    Consider the sorted naive shifted parking function 
    \[
    (p,\overline{\sigma}) = ((1,1,1,1,2,4,4,4,4,5,5,6), (-1,1,0,1,1,-1,0,0,0,0,0,0)) 
    \]
    Then $\upsilon = (2,1,0,2,2,1,0,0,0,0,0,0)$. The lattice path associated to $(p, \overline{\sigma})$ is shown in Figure~\ref{fig:lattice-path}. The matching $\tau(L)$ consists of $(1,3), (4,8), (5,7)$. Since $\upsilon_3= \upsilon_8 = \upsilon_7 = 0$, $(p, \overline{\sigma})$ is a garage. There are four sorted naive shifted parking functions with the same lattice path, obtained by replacing the four 1's in $p$ with two 1's and two 3's or replacing the four 4's in $p$ with two 4's and two 8's. In other words,
    \[
    \varphi^{-1}(p, \overline{\sigma}) = \begin{Bmatrix*}[l] ((1,1,1,1,2,4,4,4,4,5,5,6), (-1,1,0,1,1,-1,0,0,0,0,0,0))\\
    ((1,1,1,1,2,4,4,5,5,6,8,8), (-1,1,0,1,1,-1,0,-1,0,0,0,0))\\
    ((1,1,2,3,3,4,4,4,4,5,5,6), (-1,1,1,1,1,-1,0,0,0,0,0,0))\\
    ((1,1,2,3,3,4,4,5,5,6,8,8), (-1,1,1,1,1,-1,0,-1,0,0,0,0))
    \end{Bmatrix*}
    \]
    The shapes of these four preimages are $(4,4,2,1,1), (4,2,2,2,1,1), (4,2,2,2,1,1)$ and $(2,2,2,2,2,1,1)$ respectively, so by Lemma~\ref{l:naive-decomp}, 
    \[
    \ch(\CC[\varphi^{-1}(p,\overline{\sigma})]) = V_{4,4,2,1} + 2V_{4,2,2,2,1} + V_{2,2,2,2,2,1} =  (P_4+P_2P_2)(P_4+P_2P_2)P_2P_1P_1 = R_{4,4,2,1,1}
    \]
    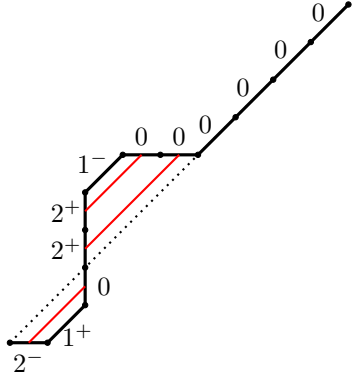
\begin{figure}[h]
    \begin{center}
    \begin{tikzpicture}[scale = .5]
         \draw[thick,dotted] (0,0)--(9,9);

         \draw[very thick] (0,0)--(1,0)--(2,1)--(2,2)--(2,3)--(2,4)--(3,5)--(4,5)--(5,5)--(6,6)--(7,7)--(8,8)--(9,9);
        \draw[fill=black] (0,0) circle (2pt);
        \draw[fill=black] (1,0) circle (2pt);
        \draw[fill=black] (2,1) circle (2pt);
        \draw[fill=black] (2,2) circle (2pt);
        \draw[fill=black] (2,3) circle (2pt);
        \draw[fill=black] (2,4) circle (2pt);
        \draw[fill=black] (3,5) circle (2pt);
        \draw[fill=black] (4,5) circle (2pt);
        \draw[fill=black] (5,5) circle (2pt);
        \draw[fill=black] (6,6) circle (2pt);
        \draw[fill=black] (7,7) circle (2pt);
        \draw[fill=black] (8,8) circle (2pt);
        \draw[fill=black] (9,9) circle (2pt);
        
        \draw[thick, red] (.5,0)--(2,1.5);
        \draw[thick, red] (2,2.5)--(4.5,5);
        \draw[thick, red] (2,3.5)--(3.5,5);

        \node at (.5,-.5) {$2^-$};
        \node at (1.8,.2) {$1^+$};
        \node at (2.5,1.5) {$0$};
        \node at (1.5,2.5) {$2^+$};
        \node at (1.5,3.5) {$2^+$};
        \node at (2.2,4.8) {$1^-$};
        \node at (3.5,5.5) {$0$};
        \node at (4.5,5.5) {$0$};
        \node at (5.2,5.8) {$0$};
        \node at (6.2,6.8) {$0$};
        \node at (7.2,7.8) {$0$};
        \node at (8.2,8.8) {$0$};
    \end{tikzpicture}
    \end{center}
    \caption{The lattice path $L$ associated to the garage in Example~\ref{ex:lattice-path}. Steps are labelled by $\upsilon$ with $\overline{\sigma}$ in the superscripts. The matching $\tau(L)$ is shown in red.}
    \label{fig:lattice-path}
    \end{figure}

\end{example}

Combining Lemma~\ref{l:one-garage-equiv} and Proposition~\ref{p:garage-R}, we have:

\begin{theorem}
    \label{t:garage-R}
    For each $n$, we have
\[
SH_n = \sum_{(p,\overline{\sigma}) \in \Gar(n)} R_{\lambda(p)}.
\]
\end{theorem}

\subsection{Odd shifted parking functions}
\label{ss:odd}
We now introduce odd shifted parking functions, which are the key objects in Theorem~\ref{t:odd}.

\begin{definition}
\label{d:odd}
    An {\em odd shifted parking function} is a triple $(p, \sigma, \tau)$, where $p$ is a parking function whose shape is odd, $\sigma \in \{-1,1\}^n$, and $\tau$ is a noncrossing matching of the integers appearing in $p$ satisfying:
    \begin{enumerate}
        \item If $(a,b) \in \tau$, then $\overline{\sigma}_a = -\overline{\sigma}_b$
        \item If $(a,c) \in \tau$, and $a<b<c$, then $b$ appears in $p$.
        \item If $(a,d) \in \tau$, $(b,c) \in \tau$ and $a<b<c<d$, then $\overline{\sigma}_a = \overline{\sigma}_b$
    \end{enumerate}
where
\[
\overline{\sigma}_k = \begin{cases}
    \prod_{p_i = k} \sigma_i & \alpha_k(p) > 0\\
    0 & \alpha_k(p) = 0
\end{cases}
\] is defined as it was for naive shifted parking functions.

For $(p, \sigma, \tau)$ an odd parking function, $\sort(p,\sigma, \tau)$ is the triple $(\sort(p), \overline{\sigma}, \tau)$, 
A \emph{sorted odd parking function} is a triple $\sort(p,\sigma, \tau)$ of some odd parking function.
Let $\OSh(n)$ be the set of odd shifted parking functions of size $n$ and $\overline{\OSh}(n)$ be the set of sorted odd shifted parking functions of size $n$.
\end{definition}

Let $\fkS_n$ act on $(p,\sigma, \tau) \in \OSh(n)$ by $\pi \cdot (p, \sigma, \tau) = (\pi \cdot (p,\sigma), \tau)$, i.e. $\fkS_n$ acts as in the naive case on $p$ and $\sigma$ and leaves $\tau$ fixed.
Since the action of $\fkS_n$ does not change $\sort(p,\sigma, \tau)$, $\mathbb{C}[\OSh(n)]$ decomposes into submodules indexed by sorted odd shifted parking functions.
For $(p,\overline{\sigma},\tau)$ a sorted odd shifted parking function, define
\begin{equation}
    \label{eq:odd-class}
O_{(p, \overline{\sigma}, \tau)} = \{ (p',\sigma', \tau') \in \OSh(n) \mid \sort(p',\sigma',\tau') = (p,\overline{\sigma},\tau)\}.
\end{equation}
Necessarily $\tau' = \tau$ in~\eqref{eq:odd-class}.
\begin{lemma}
\label{l:odd-decomp}
    Let $(p, \overline{\sigma}, \tau)$ be a sorted odd shifted parking function where $p$ has shape $\lambda$.
    Then $\mathbb{C}[O_{(p, \overline{\sigma}, \tau)}]$ is an $\mathfrak{S}_n$-submodule of $\mathbb{C}[\OSh(n)]$ with character $V_\lambda$.
\end{lemma}

\begin{proof}
    Since the $\fkS_n$--action on odd shifted parking functions leaves $\tau$ fixed, $\mathbb{C}[O_{(p, \overline{\sigma}, \tau)}] \cong \mathbb{C}[N_{(p, \overline{\sigma})}]$ and the result follows from Lemma~\ref{l:naive-decomp}
\end{proof}

To show the Frobenius character of $\mathbb{C}[\OSh(n)]$ is $SH_n$, we construct an analogue of $\varphi$ mapping  odd shifted parking functions to garages.
From this, the result will follow by Theorem~\ref{t:garage-R}.
To each odd shifted parking function we associate a garage in the following way.
\begin{definition}
    \label{d:odd-to-R}
    Let $(p, \sigma, \tau) \in \OSh(n)$.
    Construct $p'$ where:
    \begin{itemize}
        \item if $p_i =a$ and $a$ is unmatched or a left endpoint of an arc $(a,b)$ in $\tau$, then $p'_i=a$;
        \item if $p_i = b$ and $b$ is the right endpoint of an arc $(a,b)$ in $\tau$, then $p'_i = a$.
    \end{itemize}
    Here $\overline{\sigma}'$ is the restriction of $\overline{\sigma}$ to the set of integers appearing in $p'$.
    Define $\varphi_o: \OSh(n) \rightarrow \Gar(n)$ by
    \[
    \varphi_o(p, \sigma, \tau)  = (\sort(p'), \overline{\sigma}').
    \]
\end{definition}
Note for $(p,\sigma,\tau) \in \OSh(n)$ that the tuple $\upsilon$ associated to $\varphi_o((p,\sigma,\tau))$ is defined by
\begin{equation}
    \upsilon_i = \begin{cases} 2 & i \textrm{ is a left endpoint in } \tau,\\ 1 & i \textrm{ is unmatched } \tau,\\0 & i \textrm{ is a right endpoint in } \tau.\end{cases}
\label{eq:odd-upsilon}
\end{equation}

\begin{proposition}
    The map $\varphi_o$ is well defined.
\end{proposition}

\begin{proof}
    Consider $\varphi_o$ applied to $(p,\overline{\sigma},\tau) \in \OSh(n)$ with $\upsilon$ as in~\eqref{eq:odd-upsilon}.
    If $\upsilon_i =2$, then $i$ was the left endpoint of an arc $(i,j) \in \tau$ so $\alpha_i(p') = \alpha_i(p) + \alpha_j(p)$.
    Since the shape of $p$ is odd, this quantity is a positive even integer.
    If $\upsilon_i =1$, then $i$ was unmatched in $\tau$ so $\alpha_i(p') = \alpha_i(p)$. 
    If $\upsilon_i=0$, then $i$ was a right endpoint of an arc in $\tau$ so $\alpha_i(p') = 0$.
    
    All that remains is to demonstrate $(p',\overline{\sigma}')$ satisfies the lattice word property.
    To see this, consider $i<j$ with $\upsilon_i = \upsilon_j =2$ so that the word $\upsilon_{i}\upsilon_{i+1} \cdots \upsilon_{j-1} $ has as many $2$'s as $0$'s in any prefix.
    Since the $2$'s correspond to left end points and the $0$'s to right end points in $\tau$, there must exist arcs $(i,\ell) \in \tau$ and $(j,k) \in \tau$.
    Moreover we see $\ell > j$.
    Since $\tau$ is non-crossing, we have $i<j<k<\ell$ and thus $\overline{\sigma}(i) = \overline{\sigma}(j)$.
    This implies the lattice word condition, so the image of $\varphi_o$ is a garage as desired.
\end{proof}

\begin{proposition}
\label{p:odd-R}
    Let $(p, \overline{\sigma}) \in \Gar(n)$ with shape $\lambda$. Then $\mathbb{C}[\varphi_o^{-1}(p, \overline{\sigma})]$ is an $\mathfrak{S}_n$-submodule of $\mathbb{C}[\OSh(n)]$ with character \[
    \ch(\mathbb{C}[\varphi_o^{-1}(p, \overline{\sigma})]) = R_\lambda
    \]
\end{proposition}

\begin{proof}
   As in the proof of Proposition~\ref{p:garage-R} $\varphi_o $ factors through the sort map, so we can define $\varphi_o = \overline{\varphi}_o \circ \sort$.
    By Lemma~\ref{l:odd-decomp} it suffices to show that $\overline{\varphi}_o^{-1}((p, \overline{\sigma}))$ is in bijection with terms in the $\{V_\lambda:\lambda \textrm{ odd}\}$--basis expansion of $R_\lambda$.
    But this is clear, since for $(p', \overline{\sigma}', \tau) \in \overline{\varphi}_o^{-1}(p,\overline{\sigma})$ the following are determined: $\overline{\sigma}'$, $\tau$, $\alpha_c(p')$ for $c$ unmatched in $\tau$ and the quantity $2k = \alpha_a(p') + \alpha_b(p')$ for each pair $(a,b) \in \tau$.
    The only choice comes from how many $a$'s appear, and this corresponds to a choice of a term in the sum
    \[
    R_{2k} = \sum_{i=1}^k P_{2i-1}\cdot P_{2(k-i)+1}
    \]
    By Lemma~\ref{l:path-pigeonhole}, each choice corresponds to a valid sorted odd shifted parking function so the result follows by making all possible choices for each matched pair.
\end{proof}

\begin{example}
\label{ex:odd-to-garage}
    Continuing Example~\ref{ex:lattice-path}, consider the garage
    \[
    (p,\overline{\sigma}) = ((1,1,1,1,2,4,4,4,4,5,5,6), (-1,1,0,1,1,-1,0,0,0,0,0,0)) 
    \]
    There are four sorted odd shifted parking functions which map to $(p,\overline{\sigma})$, obtained by replacing the four 1's in $p$ with three 1's and a 3 (or one 1 and three 3's), replacing the four 4's in $p$ with three 4's and an 8 (or a 4 and three 8's), replacing the two 5's with a 5 and a 7, then recording the matching $\tau$. In other words,
    \[
    \varphi^{-1}(p, \overline{\sigma}) = \begin{Bmatrix*}[l] ((1,1,1,2,3,4,4,4,5,6,7,8), (-1,1,0,1,1,-1,0,0,0,0,0,0), \{(1,3),(4,8),(5,7)\})\\
    ((1,1,1,1,2,4,5,6,7,8,8,8), (-1,1,0,1,1,-1,0,-1,0,0,0,0), \{(1,3),(4,8),(5,7)\})\\
    ((1,2,3,3,3,4,4,4,5,6,7,8), (-1,1,1,1,1,-1,0,0,0,0,0,0), \{(1,3),(4,8),(5,7)\})\\
    ((1,2,3,3,3,4,5,6,7,8,8,8), (-1,1,1,1,1,-1,0,-1,0,0,0,0), \{(1,3),(4,8),(5,7)\})
    \end{Bmatrix*}
    \]
    The shapes of these four preimages are all $(3,3,1,1,1,1,1,1)$, so by Lemma~\ref{l:odd-decomp}, 
    \[
    \ch(\CC[\varphi^{-1}(p,\overline{\sigma})]) = 4V_{3,3,1,1,1,1,1,1} =  (2P_3P_1)(2P_3P_1)(P_1P_1)P_1P_1 = R_{4,4,2,1,1}
    \]

\end{example}

\begin{proof}[Proof of Theorem~\ref{t:odd}]
    By Theorem~\ref{t:garage-R} and Proposition~\ref{p:odd-R}, the result follows.
\end{proof}

As in the naive case, we can also define an action of the double cover $\fkS_n^-$ on $\CC[\OSh]$. Much as the exterior algebra decomposes into its even and odd parts, a Clifford algebra on an odd number of generators decomposes as $\cl_{2k+1} \cong \cl_{2k+1}^{1} \oplus \cl_{2k+1}^{-1}$, where $\cl_{2k+1}^{1}$ denotes the even subalgebra of $\cl_{2k+1}$ and $\cl_{2k+1}^{-1}$ denotes the odd part.

Identify $\CC[\OSh]$ with $\bigoplus_{(p, \overline{\sigma}, \tau) \in \overline{\OSh}} \CC[\fkS_n^-] \otimes_{\CC[\fkS_{\alpha(p)}^-]} \cl_{\alpha(p)_1}^{\overline{\sigma}_1} \otimes \cdots \otimes \cl_{\alpha(p)_{n}}^{\overline{\sigma}_n}$ via
\[
(p, \sigma, \tau) \mapsto \pi \otimes \xi_{i_1}\cdots \xi_{i_k}
\]
 where $\pi \cdot p = \sort(p)$ and $\{i_1 \leq \cdots \leq i_k\}$ is the set of indices $i$ for which $\sigma_{\pi(i)} = -1$. We can consider $\CC[\OSh]$ to be a $\fkS_n^-$ module under this identification, and we have the following.
\begin{corollary}
    \label{c:odd-projective}
    The spin character of $\CC[\OSh]$ is $2^{n/2}SH_n$
\end{corollary}

\begin{proof}
    This follows from Theorem~\ref{t:odd}, Proposition~\ref{p:char-of-clifford}, and \cite[Theorem 5.3]{stembridge}.
\end{proof}

\section{Problems and Extensions}
\label{s:final}

\subsection{Enumerative aspects}
\label{ss:final-enumerative}

From the definition, it is easy to see
\[
\#\{(p,\overline{\sigma}) \in \NSh(n):\lambda(p) = \mu\} = 2^\ell(\mu) \Krew(\mu).
\]
Since there are direct combinatorial arguments showing $\Krew(\mu)$ counts sorted parking functions with shape $\mu$, this identity has a combinatorial proof.
By contrast, we have the open problem:
\begin{problem}
    \label{prob:odd-krew}
    For $\mu \vdash n$ odd, give a direct combinatorial argument showing
\[
\OKrew(\mu) = \# \{(p,\overline{\sigma}) \in \overline{\OSh}(n): \lambda(p) = \mu\}.
\]
\end{problem}

One potentially useful fact for tackling Problem~\ref{prob:odd-krew} is the observation:

\begin{proposition}
    \label{p:odd-krew}
    For $\mu \vdash n$ odd and $\ell = \ell(\mu)$, we have
    \[
    \OKrew(\mu) = \Krew(\mu) \cdot {\binom{n+\ell}{\frac{n+\ell}{2}}}{\Bigg 
    /} {\binom{n}{\frac{n+\ell}{2}}}
    \]
\end{proposition}

\begin{proof}
Let $m_i = m_i(\mu)$.
Then
    \begin{align*}
\OKrew(\mu) &= \frac{2^{\ell}}{\ell!}\binom{\ell}{m_1, m_3, \dots}
(n + \ell - 1)(n + \ell - 3)\cdots(n - \ell + 3) \\&= \frac{2^{\ell}}{\ell!}\binom{\ell}{m_1, m_3, \dots}
\frac{(n+\ell)!(\frac{n-\ell}{2})!}{2^{\ell}\frac{n+\ell}{2}!(n-\ell+1)!}\\&=\frac{1}{n+1} \binom{n+1}{n-\ell+1,m_1, m_3, \dots}
\frac{(n+\ell)!(\frac{n-\ell}{2})!}{\frac{n+\ell}{2}!n!}
\\&= \textrm{Krew}(\lambda)
\frac{(n+\ell)!(\frac{n-\ell}{2})!}{n!\frac{n+\ell}{2}!} =\textrm{Krew}(\lambda)
{\binom{n+\ell}{\frac{n+\ell}{2}}}{\Bigg /}{\binom{n}{\frac{n+\ell}{2}}}.
\end{align*}
\end{proof}

As mentioned previously, Stanley observed $\# \overline{\NSh}(n)$ is the $n$th \emph{large Schr\"oder number} $s_n$, which is defined by the initial conditions and recurrence
\[
s_0 = 1, \quad s_1 = 2, \quad s_{n} = 3s_{n-1} + \sum_{k=1}^{n-2} s_k s_{n-k-1}\ (n \geq 2).
\]
Therefore $\# \overline{\OSh}(n) = s_n$ as well.
The large Schr\"oder number $s_n$ enumerates \emph{Schr\"oder paths} of length $n$, which are lattice paths from $(0,0)$ to $(n,n)$ with steps $(1,0)$, $(0,1)$ and $(1,1)$ staying weakly above the line $y=x$.
There is a straightforward bijection from $\overline{\NSh}(n)$ to Schr\"oder paths extending the classical bijection from sorted parking functions to Dyck paths.
For each positive integer $i$
\begin{itemize}
    \item if $\overline{\sigma}_i = 1$ or $0$, add $\alpha_i(p)$ steps $(0,1)$, followed a step $(1,0)$;
    \item if $\overline{\sigma}_i = -1$, add $\alpha_i(p)-1$ steps $(0,1)$ followed by a step $(1,1)$.
\end{itemize}

Our proof of Theorem~\ref{t:odd} can be used to give a bijection between odd and naive sorted shifted parking functions, as we now explain.
Given a sorted naive shifted parking function $(p,\overline{\sigma})$ with auxiliary matching $\tau$, let $p'$ be obtained from $p$ by replacing exactly one $i$ with a $j$ for every $(i,j) \in \tau$.
Let $\overline{\sigma}'$ be obtained from $\overline{\sigma}$ by setting $\overline{\sigma}'_j = -\overline{\sigma}_i$ if $(i,j) \in \tau$ and $\overline{\sigma}'_j = \overline{\sigma}_j$ otherwise. Then the map
\begin{equation}
(p, \overline{\sigma}) \mapsto (p',\overline{\sigma}',\tau)
\label{eq:bijection}
\end{equation}
defines a bijection from $\overline{\NSh}$ to $\overline{\OSh}$.

We were unable to find any objects enumerated by large Schr\"oder numbers placing a similar emphasis on odd parity in its constituent components, suggesting $\overline{\OSh}(n)$ could be the first instance in a new and combinatorially distinct family of large Schr\"oder objects.

\subsection{Algebraic aspects}
\label{ss:final-algebraic}

As discussed in the introduction, a major motivation for studying the parking function symmetric function is its bigraded form
\begin{equation}
\label{eq:bigraded}
PF_n(q,t) = \sum_{p \in \Pf(n)} t^{\area(p)} q^{\dinv(p)} F_{n, \textrm{ides}(p)}.
\end{equation}
Here $\dinv$ and $\area$ are $\NN$--valued statistics on $p$ and $\textrm{ides}$ is a set-valued statistic on $p$.
The Shuffle conjecture, now a theorem~\cite{carlsson2018proof}, says $PF_n(q,t) = \nabla e_n$ where $\nabla$ is Bergeron and Garsia's celebrated \emph{nabla operator}~\cite{bergeron1999science}.
Define
\[
C_n(q,t) = \sum_{p \in \overline{\Pf}(n)} t^{\area(p)} q^{\dinv(p)}
\]
As Haglund and Garsia showed~\cite{garsia2001positivity}, $C_n(q,t)$ is related to $\nabla e_n$ by
\[
\langle \nabla e_n, e_n \rangle = C_n(q,t),
\]
where $\langle \cdot,\cdot \rangle$ is the Hall inner product on symmetric functions.
Equivalently, $C_n(q,t)$ is the coefficient of $s_{1^n}$ of $\nabla e_n$.
In~\cite{haiman2001hilbert}, Haiman showed $\nabla e_n$ is the bi-graded Frobenius character for the $\fkS_n$--module of \emph{diagonal harmonics}
\[
\mathcal{D}_n = 
(\CC[x_1,\dots,x_n]\otimes\CC[y_1,\dots,y_n])/ \mathcal{I}_n
\]
where $\mathcal{I}_n$ is the degree $\geq1$ invariants under the diagonal $\fkS_n$--action.

Applying shiftification, we obtain a $(q,t)$--analogue $SH_n(q,t) = \sh PF_n(q,t)$ of the shifted parking function symmetric function.
By Proposition~\ref{p:shiftification}, we have:
\begin{proposition}\label{p:q-t-shifted}
    For $n \geq 1$, $SH_n(q,t)$ is the bigraded Frobenius character of $(D_n \otimes \Lambda_n)$.
\end{proposition}

Note, in the natural presentation of $D_n \otimes \Lambda_n$, we now have two sets of commuting variables and one set of anti-commuting variables modulo the $\fkS_n$--invariants from the diagonal action on the commuting variables.
This is somewhat reminiscent of the \emph{super diagonal coinvariant algebra} introduced in~\cite{zabrocki2019module}, and more generally the $(r,s)$--\emph{bosonoic-fermionic coinvariant algebra} $R^{(r,s)}_n$ from~\cite{bergeron2020bosonic}, which has $r$ sets of commuting variables and $s$ sets of anti-commuting variables, modulo the space of $\fkS_n$--invariants.
However, while Proposition~\ref{p:q-t-shifted} is closely related to $R^{(2,1)}_n$, we are not modding out invariants in the anticommuting variables.

We will see $SH_n(q,t)$ is closely related to Haglund's results on Schr\"oder paths.
Let $\mathrm{S}(n)$ be the set of Schr\"oder paths of size $n$ and $\mathrm{S}_k(n)$ be the subset of Schr\"oder paths with $k$ horizontal steps.
In~\cite{egge2003schroder} statistics $\area$ and $\bounce$ on Schr\"oder paths were introduced so that:
\begin{theorem}[{\cite{haglund2004proof}}]
    For $k \in [0,n]$,
\begin{equation}
    \langle\nabla e_n, h_ke_{n-k}\rangle = \sum_{P \in \mathrm{S}_k(n)} q^{\area(P)} t^{\bounce(P)}.
\label{eq:haglund}
\end{equation}
\end{theorem}

Summing~\eqref{eq:haglund} over all values $k$, then applying Proposition~\ref{p:self-adjoint} and~\eqref{eq:P-def}, we obtain
\begin{equation}
\langle\nabla e_n, 2P_n\rangle = \langle \sh \nabla e_n, h_n\rangle = \langle SH_n(q,t), h_n\rangle = \sum_{s \in \mathrm{S}(n)} q^{\area(s)} t^{\bounce(s)}.
\label{eq:nabla-P}    
\end{equation}
Schr\"oder paths can be identified with sorted naive shifted parking functions as discussed in Section~\ref{ss:final-enumerative}.
Therefore, we can interpret~\eqref{eq:nabla-P} as a statement about sorted naive shifted parking functions with appropriate statistics.

It is natural to ask:
\begin{problem}
\label{prob:odd-nabla}
    Are there statistics $\stat_1$ and $\stat_2$ on sorted odd shifted parking functions so that
\[
\langle \nabla e_n , 2P_n \rangle = \langle SH_n(q,t),h_n\rangle =  \sum_{(p,\overline{\sigma},\tau) \in \overline{\OSh}(n)} q^{\stat_1(p,\overline{\sigma},\tau)} t^{\stat_2(p,\overline{\sigma},\tau)}?
\]
\end{problem}

Alternatively, it is plausible that $SH_n$ has a different $q,t$--analogue more naturally associated to odd shifted parking functions.
Perhaps this would be based on a bi-grading of the projective representation associated to odd shifted parking functions.
It would then be natural to ask:

\begin{problem}
    \label{prob:shifted-harmonics}
Is there another $\fkS_n$--module (or $\overline{\fkS}_n$--module) analogous to the diagonal harmonics whose bi-graded Frobenius character specializes to $SH_n$ when $q = t =1$?
\end{problem}

It would also be interesting to find combinatorial interpretations for the expansion of $SH_n(q,t)$ into bases in $\SymP$.
Depending on the choice of basis, this could be quite challenging since finding explicit expansions of $PF_n(q,t)$ is an open problem for most bases of $\Sym$.

Though finding the $e$--expansion of $\nabla e_n$ is an open problem,
by setting $q = 1$ we have
\[
PF_n(1,t) = \sum_{p \in \overline{\Pf}(n)} t^{\area(p)} e_{\lambda(p)}.
\]
Therefore, applying shiftification gives
\begin{equation}
\label{eq:t-sh}
SH_n(1,t) = \sum_{(p,\overline{\sigma}) \in \overline{\NSh}(n)} t^{\area(p)} V_{\lambda(p)}.
\end{equation}
The area of a parking function can be computed as
\[
\area(p) = \sum_{i=1}^n (p_i - i) = \left(\sum_{i=1}^n (n{+}1{-}i)\cdot \alpha_i(p)\right) - \binom{n}{2}.
\]
In the bijection from $\overline{\NSh}(n) \to \overline{\OSh}(n)$ defined in~\eqref{eq:bijection}, we see for $(p,\overline{\sigma}) \in \overline{\NSh}(n)$ with auxiliary matching $\tau$ mapping to $(p',\overline{\sigma}',\tau) \in \overline{\OSh}(n)$ that
\begin{equation}
    \label{eq:odd-area}
\area(p) = \area(p') + \sum_{(i,j) \in \tau} j-i.
\end{equation}
By defining $\area_o(p',\overline{\sigma}',\tau)$ as the righthand side of~\eqref{eq:odd-area}, Theorem~\ref{t:odd} extends to the $q=1$ setting.
\begin{proposition}
    \label{p:t-odd-sh}
    \[
    SH_n(1,t) = \sum_{(p,\overline{\sigma},\tau)} t^{\area_o(p,\overline{\sigma},\tau)} V_{\lambda(p)}.
    \]
\end{proposition}

\begin{proof}
    The result follows from~\eqref{eq:t-sh} and~\eqref{eq:odd-area}.
\end{proof}
    
We suspect the relatively simple combinatorial description  of $SH_n(1,t)$ in the $V_\lambda$ basis is evidence of some deeper underlying algebraic story.
Our hope is that odd shifted parking functions will lead to further developments in the story of diagonal harmonics and related algebras.

\bibliographystyle{plain}
\bibliography{references.bib}

\end{document}